\documentclass{article}
\usepackage{subfigure,amsthm}
\usepackage{tikz}
\usetikzlibrary{arrows,decorations.pathmorphing}
\author{
Masaki Watanabe \\
Graduate School of Mathematics, \\
the University of Tokyo \\
\texttt{mwata@ms.u-tokyo.ac.jp}
\footnote{This work was supported by Grant-in-Aid for JSPS Fellows No.~15J05373. }
}
\title{Kra\'skiewicz-Pragacz modules and Pieri and dual Pieri rules for Schubert polynomials}

\usepackage{amssymb,amsmath}
\newcommand{\NN}{\mathbb{N}}
\newcommand{\ZZ}{\mathbb{Z}}
\newcommand{\nonneg}{\ZZ_{\geq 0}}

\newtheorem{lem}{Lemma}[section]
\newtheorem{prop}[lem]{Proposition}
\newtheorem{thm}[lem]{Theorem}

\theoremstyle{definition}
\newtheorem{defn}[lem]{Definition}
\newtheorem{rmk}[lem]{Remark}
\newtheorem{eg}[lem]{Example}
\newcommand{\schub}{\mathfrak{S}}
\newcommand{\smod}{\mathcal{S}}

\newcommand{\borel}{\mathfrak{b}}
\newcommand{\mor}{\rightarrow}
\newcommand{\surj}{\twoheadrightarrow}

\newcommand{\lex}[1]{{ \,\underset{\mathrm{lex}}{#1}\, }}
\newcommand{\rlex}[1]{{ \,\underset{\mathrm{rlex}}{#1}\, }}
\newcommand{\lcov}{\lessdot}
\newcommand{\gcov}{\gtrdot}
\newcommand{\id}{\mathrm{id}}
\newcommand{\der}{\partial}
\newcommand{\lmb}{\lambda}
\newcommand{\Lmb}{\Lambda}
\newcommand{\catc}{\mathcal{C}}
\newcommand{\ch}{\mathrm{ch}}
\DeclareMathOperator{\sgn}{sgn}
\begin{document}
\maketitle
\paragraph{Abstract.}
In their 1987 paper Kra\'skiewicz and Pragacz defined certain modules, which we call KP modules, over the upper triangular Lie algebra whose characters are Schubert polynomials. 
In a previous work the author showed that the tensor product of KP modules always has a KP filtration, i.e. a filtration whose each successive quotients are isomorphic to KP modules. 
In this paper we explicitly construct such filtrations for certain special cases of these tensor product modules, namely $\smod_w \otimes S^d(K^i)$ and $\smod_w \otimes \bigwedge^d(K^i)$, corresponding to Pieri and dual Pieri rules for Schubert polynomials. 



\section{Introduction}
Schubert polynomials are one of the main subjects in algebraic combinatorics. 
One of the tools for studying Schubert polynomials is the modules introduced by Kra\'skiewicz and Pragacz. These modules, which here we call \emph{KP modules}, are modules over the upper triangular Lie algebra and have the property that their characters with respect to the diagonal matrices are Schubert polynomials. 

It is known that a product of Schubert polynomials is always a positive sum of Schubert polynomials. The previously known proof for this positivity property uses the geometry of the flag variety. In \cite{W} the author showed that the tensor product of two KP modules always has a filtration by KP modules and thus gave a representation-theoretic proof for this positivity. Although the proof there does not give explicit constructions for the KP filtrations, it may provide a new viewpoint for the notorious Schubert-LR problem asking for a combinatorial positive rule for the coefficient in the expansion of products of Schubert polynomials into a sum of Schubert polynomials. 

There are some cases where the expansions of products of Schubert polynomials are explicitly known. Examples of such cases include the Pieri and the dual Pieri rules for Schubert polynomials (\cite{BB}, \cite{Las}, \cite{Sot}, \cite{Win}). 
These are the cases where one of the Schubert polynomials is a complete symmetric function $h_d(x_1, \ldots, x_i)$ or an elementary symmetric function $e_d(x_1, \ldots, x_i)$. 
The purpose of this paper is to investigate the structure of tensor product modules corresponding to these products and to give explicit constructions of KP filtrations for these modules. 

The structure of this paper is as follows. In Section \ref{prelimi} we prepare some definitions and results on Schubert polynomials and KP modules. In Section \ref{sec:pieri-schub} we review the Pieri and the dual Pieri rules for Schubert polynomials. In Section \ref{sec:pieri-kp} we give explicit constructions for KP filrartions of the corresponding tensor product modules $\smod_w \otimes S^d(K^i)$ and $\smod_w \otimes \bigwedge^d(K^i)$. In Section \ref{sec:pieri-kp-prf} we give a proof of the main result.

\section{Preliminaries}
\label{prelimi}
Let $\NN$ be the set of all positive integers. 
By a \emph{permutation} we mean a bijection from $\NN$ to itself which fixes all but finitely many points. 
The \emph{graph} of a permutation $w$ is the set $\{(i, w(i)) : i \in \NN\} \subset \NN^2$. 
For $i < j$, let $t_{ij}$ denote the permutation which exchanges $i$ and $j$ and fixes all other points. Let $s_i = t_{i,i+1}$. 
For a permutation $w$, let $\ell(w)=\#\{i<j : w(i)>w(j)\}$. 
For a permutation $w$ and positive integers $p < q$, if $\ell(wt_{pq}) = \ell(w)+1$ we write $wt_{pq} \gcov w$. 
It is well known that this condition is equivalent to saying that $w(p)<w(q)$ and there exists no $p < r < q$ satisfying $w(p)<w(r)<w(q)$. 
For a permutation $w$ let $I(w) = \{(i,j) : i<j, w(i)>w(j)\}$. 

For a polynomial $f=f(x_1, x_2, \ldots)$ and $i \in \NN$ define $\der_if=\frac{f-s_if}{x_i-x_{i+1}}$. 
For a permutation $w$ we can assign its \emph{Schubert polynomial} $\schub_w \in \ZZ[x_1, x_2, \ldots]$ which is recursively defined by 
\begin{itemize}
\item $\schub_{w}=x_1^{n-1}x_2^{n-2} \cdots x_{n-1}$ if $w(1)=n, w(2)=n-1, \ldots, w(n)=1$ and $w(i)=i$ ($i > n$) for some $n$, and
\item $\schub_{ws_i}=\der_i\schub_w$ if $\ell(ws_i)<\ell(w)$. 
\end{itemize}

Hereafter let us fix a positive integer $n$. 
Let \[S^{(n)} = \{\text{$w$ : permutation, $w(n+1) < w(n+2) < \cdots$}\}.\] 
Note that if $w \in S^{(n)}$ then $I(w) \subset \{1, \ldots, n\} \times \NN$. 
Let $K$ be a field of characteristic zero. 
Let $\borel = \borel_n$ denote the Lie algebra of all $n \times n$ upper triangular matrices over $K$. 
For a $\borel$-module $M$ and $\lmb = (\lmb_1, \ldots, \lmb_n) \in \ZZ^n$, 
let $M_\lmb = \{m \in M : hm=\langle \lmb,h \rangle m \;\mbox{($\forall h = \mathrm{diag}(h_1, \ldots, h_n)$)}\}$ where $\langle \lmb,h \rangle = \sum_i \lmb_i h_i$. 
If $M$ is a direct sum of these $M_\lmb$ and these $M_\lmb$ are finite dimensional then we say that $M$ is a \emph{weight module}
and we define $\ch(M)=\sum_{\lmb} \dim M_\lmb x^\lmb$ where $x^\lmb=x_1^{\lmb_1}  \cdots x_n^{\lmb_n}$. 
For $1 \leq i \leq j \leq n$ let $e_{ij} \in \borel$ be the matrix with $1$ at the $(i,j)$-th position and all other coordinates $0$. 

Let $U$ be a vector space spanned by a basis $\{u_{ij} : 1 \leq i \leq n, j \in \NN \}$. Let $T = \bigoplus_{d=0}^\infty \bigwedge^d U$. 
The Lie algebra $\borel$ acts on $U$ by $e_{pq} u_{ij} = \delta_{iq}e_{pj}$ and thus on $T$. 
For $w \in S^{(n)}$ let $u_w = \bigwedge_{(i,j) \in I(w)} u_{ij} \in \bigwedge^{\ell(w)} U \subset T$. 
The \emph{Kra\'skiewicz-Pragacz module} $\smod_w$ (or the \emph{KP module} for short) associated to $w$ is the $\borel$-submodule of $\bigwedge^{\ell(w)} U \subset T$ generated by $u_w$. 
In \cite{KP} Kra\'skiewicz and Pragacz showed the following: 
\begin{thm}[{{\cite[Remark 1.6 and Theorem 4.1]{KP}}}]
$\smod_w$ is a weight module and $\ch(\smod_w)=\schub_w$. 
\end{thm}
\begin{eg}
If $w = s_i$, then $u_w = u_{i,i+1} \in U$ and thus $\smod_w = \bigoplus_{1 \leq j \leq i} Ku_{j,i+1} \cong \bigoplus_{1 \leq j \leq i}Ku_j =: K^i$ on which $\borel$ acts by $e_{pq} u_j = \delta_{qj}u_p$. 
\end{eg}

A \emph{KP filtration} of a $\borel$-module $M$ is a filtration $0 = M_0 \subset \cdots \subset M_r = M$ such that each $M_i / M_{i-1}$ is isomorphic to some KP module. 

\section{Pieri and dual Pieri rules for Schubert polynomials}
\label{sec:pieri-schub}
\begin{defn}
For $w \in S_\infty$, $i \geq 1$ and $d \geq 0$, let 
\[
X_{i,d}(w) = \{t_{p_1q_1} t_{p_2q_2} \cdots t_{p_dq_d} : p_j \leq i, q_j > i, w_1 \lcov w_2 \lcov \cdots, w_1(p_1) < w_2(p_2) < \cdots \}
\]
and 
\[
Y_{i,d}(w) = \{t_{p_1q_1} t_{p_2q_2} \cdots t_{p_dq_d} : p_j \leq i, q_j > i, w_1 \lcov w_2 \lcov \cdots, w_1(q_1) > w_2(q_2) > \cdots \}
\]
where $w_1=w, w_2=wt_{p_1q_1}, w_3=wt_{p_1q_1}t_{p_2q_2}, \cdots$.  
\label{defn-pieri}
\end{defn}

Note that the condition for $X_{i,d}(w)$ (resp.\ $Y_{i,d}(w)$) implies that $q_1, \ldots, q_d$ (resp.\ $p_1, \ldots, p_d$) are all different. 

\begin{thm}[Conjectured in \cite{BB} and proved in \cite{Win}, also appears with different formulations in \cite{Las} and \cite{Sot}]
We have
\[
\schub_w \cdot h_d(x_1, \ldots, x_i) = \sum_{\zeta \in X_{i,d}(w)} \schub_{w\zeta}
\]
and
\[
\schub_w \cdot e_d(x_1, \ldots, x_i) = \sum_{\zeta \in Y_{i,d}(w)} \schub_{w\zeta}
\]
where $h_d$ and $e_d$ denote the complete and elementary symmetric functions respectively. 
\hfill $\Box$
\label{schub-pieri}
\footnote{
The formulation of dual Pieri rule here is slightly different from the one in \cite{BB} but they are easily shown to be equivalent. 
}
\end{thm}

Note here that the permutation $\zeta \in X_{i,d}(w)$ (or $Y_{i,d}(w)$) in fact uniquely determines its decomposition into transpositions satisfying the conditions in Definition \ref{defn-pieri}. 
So we can write, without ambiguity, for example ``for $\zeta = t_{p_1q_1} \cdots t_{p_dq_d} \in X_{i,d}(w)$ define (something) as (some formula involving $p_j$ and $q_j$)''. 
Hereafter if we write such we will always assume the conditions in Definition \ref{defn-pieri}. 


\section{Explicit Pieri and dual Pieri rules for KP modules}
\label{sec:pieri-kp}
The author showed in \cite{W} that the tensor product of KP modules always has a KP filtration. Since $S^d(K^i)$ and $\bigwedge^d(K^i)$ ($1 \leq i \leq n$, $d \geq 1$) are special cases of KP modules, $\smod_w \otimes S^d(K^i)$ and $\smod_w \otimes \bigwedge^d(K^i)$ ($w \in S^{(n)}$) have KP filtrations. In this section we give explicit constructions for these filtrations. 

For positive integers $p \leq q$ we define an operator $e'_{qp}$ acting on $T$ as $e'_{qp} (u_{a_1b_1} \wedge u_{a_2b_2} \wedge \cdots) = \sum_{k}(\cdots \wedge \delta_{p b_k}u_{a_kq} \wedge \cdots)$. 
Let these operators act on $T \otimes S^d(K^i)$ and $T \otimes \bigwedge^d(K^i)$ by applying them on the left-hand side tensor component. 
Also for $j \geq 1$ define an operator $\mu_j : T \otimes \bigotimes^a(K^i) \mor T \otimes \bigotimes^{a-1}(K^i)$ ($a \geq 1$) by $u \otimes (v_1 \otimes v_2 \otimes \cdots) \mapsto (\iota_j(v_1) \wedge u) \otimes (v_2 \otimes v_3 \otimes \cdots)$ where $\iota_j(u_p) = u_{pj}$ ($1 \leq p \leq i$). We denote the restrictions of $\mu_j$ to $T \otimes S^a(K^i)$ and $T \otimes \bigwedge^a(K^i)$ (seen as submodules of $T \otimes \bigotimes^a(K^i)$) by the same symbol. 
Note that $e'_{rs}$ and $\mu_j$ give $\borel$-endomorphisms on $T \otimes S^\bullet(K^i)$ and $T \otimes \bigwedge^\bullet(K^i)$. 

For a permutation $z$ and positive integers $p<q$ let $m_{pq}(z) = \#\{r>q : z(p)<z(r)<z(q)\}$ and $m'_{qp}(z) = \#\{r<p : z(p)<z(r)<z(q)\}$. 
For $\zeta = t_{p_1q_1} \cdots t_{p_dq_d} \in X_{i,d}(w)$ (resp.\ $Y_{i,d}(w)$) define
\begin{align*}
v_{\zeta} &= (\prod_j e_{p_jq_j}^{m_{p_jq_j}(w_j)} u_w) \otimes \prod_j u_{p_j} \\
&= (\prod_j e_{p_jq_j}^{m_{p_jq_j}(w_j)} u_w) \otimes \left( \sum_{\sigma \in S_d} u_{p_{\sigma(1)}} \otimes \cdots \otimes u_{p_{\sigma(d)}}\right) \in \smod_w \otimes S^d(K^i)
\end{align*}
(resp.\
\begin{align*}
v_\zeta &= (\prod_j e_{p_jq_j}^{m_{p_jq_j}(w_j)} u_w) \otimes \bigwedge_j u_{p_j} \\
&= (\prod_j e_{p_jq_j}^{m_{p_jq_j}(w_j)} u_w) \otimes \left( \sum_{\sigma \in S_d} \sgn \sigma \cdot u_{p_{\sigma(1)}} \otimes \cdots \otimes u_{p_{\sigma(d)}}\right) \in \smod_w \otimes \bigwedge^d(K^i)
\end{align*}
) \\
where $w_j = wt_{p_1q_1}\cdots t_{p_{j-1}q_{j-1}}$ as in Definition \ref{defn-pieri}. Note that these are also well-defined even if some $q_j$ are greater than $n$, since in such a case $m_{p_jq_j}(w_j) = 0$. Note also that the products of the operators $e_{p_jq_j}$ above are well-defined since the operators $e_{p_jq_j}$ ($p_j \leq i$, $q_j > i$) commute with each others. 
Also, for such $\zeta$, define a $\borel$-homomorphism $\phi_\zeta : T \otimes \bigotimes^d(K^i) \mor T$ 
by
\[
\phi_{\zeta} = \mu_{q_d}\cdots\mu_{q_1} \cdot \prod_j (e'_{q_jp_j})^{m'_{q_jp_j}(w_j)}. \] Note that the order in the product symbol does not matter since the operators $e'_{q_jp_j}$ commute.

Let $\lex<$ and $\rlex<$ denote the lexicographic and reverse lexicographic orderings on permutations respectively, i.e.\ for permutations $u$ and $v$, $u \lex< v$ (resp.\ $u \rlex< v$)
if there exists a $k$ such that $u(j)=v(j)$ for all $j < k$ (resp.\ $j > k$) and $u(k) < v(k)$. 

\begin{prop}
For $\zeta, \zeta' \in X_{i,d}(w)$ (resp.\ $Y_{i,d}(w)$), 
\begin{itemize}
\item $\phi_{\zeta}(v_\zeta)$ is a nonzero multiple of $u_{w\zeta} \in T$, and
\item $\phi_{\zeta'}(v_\zeta) = 0$ if $(w\zeta)^{-1} \lex< (w\zeta')^{-1}$ (resp.\ $(w\zeta)^{-1} \rlex< (w\zeta')^{-1}$). 
\end{itemize}
\label{mainprop}
\end{prop}
The proof for this proposition is given in the next section. 
Here we first see that Proposition \ref{mainprop} gives desired filtrations. 

For a $\borel$-module $M$ and elements $x, y, \ldots, z \in M$ let $\langle x, y, \ldots, z \rangle$ denote the submodule generated by these elements. Consider the sequence of submodules 
\[
0 \subset \langle v_{\zeta_1} \rangle \subset \langle v_{\zeta_1}, v_{\zeta_2} \rangle 
\subset \cdots \subset
\langle v_\zeta : \text{$\zeta \in X_{i,d}(w)$ (resp.\ $Y_{i,d}(w)$)} \rangle
\]
inside $\smod_w \otimes S^d(K^i)$ (resp.\ $\smod_w \otimes \bigwedge^d(K^i)$), where $\zeta_1, \zeta_2, \ldots \in X_{i,d}(w)$ (resp.\ $Y_{i,d}(w)$) are all the elements ordered increasingly by the lexicographic (resp.\ reverse lexicographic) ordering of $(w\zeta)^{-1}$. 
From the proposition we see that there are surjections
$
\langle v_{\zeta_1}, \cdots, v_{\zeta_j} \rangle
/
\langle v_{\zeta_1}, \cdots, v_{\zeta_{j-1}} \rangle
\surj 
\smod_{w\zeta_j}
$
induced from $\phi_{\zeta_j}$. 
Thus we have
\[
\dim(\smod_w \otimes S^d(K^i))
\geq
\dim \langle v_\zeta : \text{$\zeta \in X_{i,d}(w)$} \rangle
\geq 
\sum_{\zeta \in X_{i,d}(w)} \dim \smod_{w\zeta}
=
\dim(\smod_w \otimes S^d(K^i))
\]
and
\[
\dim(\smod_w \otimes \bigwedge^d(K^i))
\geq
\dim \langle v_\zeta : \text{$\zeta \in Y_{i,d}(w)$} \rangle
\geq 
\sum_{\zeta \in Y_{i,d}(w)} \dim \smod_{w\zeta}
=
\dim(\smod_w \otimes \bigwedge^d(K^i))
\]
respectively, where the last equalities are by Proposition \ref{schub-pieri}. 
So the equalities must hold everywhere. Thus $\langle v_\zeta : \text{$\zeta \in X_{i,d}(w)$ (resp.\ $Y_{i,d}(w)$)} \rangle = \smod_w \otimes S^d(K^i)$ (resp.\ $\smod_w \otimes \bigwedge^d(K^i)$) and the surjections above are in fact isomorphisms. 
So, in conclusion, we get from Proposition \ref{mainprop} the following:
\begin{thm}
Let $M = \smod_w \otimes S^d(K^i)$ (resp.\ $\smod_w \otimes \bigwedge^d(K^i)$). Define $v_\zeta$ and $\phi_\zeta$ as above.
Then $M$ is generated by $\{ v_\zeta : \text{$\zeta \in X_{i,d}(w)$ (resp.\ $Y_{i,d}(w)$)} \}$ as a $\borel$-module and 
\[
0 \subset \langle v_{\zeta_1} \rangle \subset \langle v_{\zeta_1}, v_{\zeta_2} \rangle 
\subset \cdots \subset
\langle v_\zeta : \text{$\zeta \in X_{i,d}(w)$ (resp.\ $Y_{i,d}(w)$)} \rangle
\]
gives a KP filtration of $M$, where $\zeta_1, \zeta_2, \ldots \in X_{i,d}(w)$ (resp.\ $Y_{i,d}(w)$) are all the elements ordered increasingly by the lexicographic (resp.\ reverse lexicographic) ordering of $(w\zeta)^{-1}$. The explicit isomorphism
$
\langle v_{\zeta_1}, \cdots, v_{\zeta_j} \rangle
/
\langle v_{\zeta_1}, \cdots, v_{\zeta_{j-1}} \rangle
\cong
\smod_{w\zeta_j}
$
is given by $\phi_{\zeta_j}$ defined above.
\label{filtr}
\end{thm}

\begin{rmk}
In \cite{W2} we related KP modules with the notion of highest weight categories  (\cite{CPS}) as follows.  For $\Lmb' \subset \ZZ^n$ let $\catc_{\Lmb'}$ be the category of weight $\borel_n$-modules whose weights are all in $\Lmb'$. Then if $\Lmb'$ is an order ideal with respect to a certain ordering on $\ZZ^n$ then $\catc_{\Lmb'}$ has a structure of highest weight category whose standard objects are KP modules. 
One of the axioms for highest weight categories requires that the projective objects should have filtrations by standard objects. 

It can be shown that the projective cover of the one dimensional $\borel_n$-module $K_\lmb$ with weight $\lmb = (\lmb_1, \ldots, \lmb_n) \in \nonneg^n$ in the category $\catc_{\nonneg^n}$ is given by $S^{\lmb_1}(K^1) \otimes \cdots \otimes S^{\lmb_n}(K^n)$. 
Thus Theorem \ref{filtr} gives a proof to the fact that the indecomposable projective modules in $\catc_{\nonneg^n}$ have KP filtrations, which leads to a different proof from the one in \cite[\S 3]{W2} for the axiom mentioned above (we do not need these results about highest weight structure for $\borel$-modules in the proof of Theorem \ref{filtr} which will be done below). 
\end{rmk}

\section{Proof of Proposition \ref{mainprop}}
\label{sec:pieri-kp-prf}
\begin{lem}
Let $w \in S^{(n)}$ and $i \geq 1$. For $p, p' \leq i$ and $q, q' > i$ such that $wt_{pq}, wt_{p'q'} \gcov w$ (i.e. $t_{pq}, t_{p'q'} \in X_{i,1}(w)$), 
if $u_{pq'} \wedge e_{pq}^{m_{pq}(w)} (e'_{q'p'})^{m'_{q'p'}(w)} u_w\neq 0$ then $w(p') \geq w(p)$ and $w(q') \geq w(q)$, and if $(p, q) = (p', q')$ it is a nonzero multiple of $u_{wt_{pq}}$. 
\label{pqpq}
\end{lem}
\begin{proof}
This is essentially the same as \cite[Lemma 5.8]{W2}.
\end{proof}

\begin{lem}
Let $w$ be a permutation, $i \geq 1$ and $d \geq 0$. 
Let $\zeta = t_{p_1q_1} \cdots t_{p_dq_d} \in X_{i,d}(w)$ (resp.\ $Y_{i,d}(w)$) and $1 \leq a \leq d$. Suppose that there exists no $b < a$ satisfying $p_b = p_a$ (resp.\ $q_b = q_a$). Then $m_{p_aq_a}(w_a) = m_{p_aq_a}(w)$ and $m'_{q_ap_a}(w_a) = m'_{q_ap_a}(w)$ where $w_a = wt_{p_1q_1} \cdots t_{p_{a-1}q_{a-1}}$ as in Definition \ref{defn-pieri}. 
\label{m-invar}
\end{lem}
\begin{proof}
We show the case $\zeta \in X_{i,d}(w)$: the other case can be treated similarly. 
First note that $p_1, \ldots, p_{a-1} \neq p_a$ by the hypothesis. Also, as we have remarked before, $q_1, \ldots, q_a$ are all different. Thus the proof is now reduced to the following lemma.  
\end{proof}
\begin{lem}
Let $p<q$, $p'<q'$ and suppose
\begin{itemize}
\item $\{p, q\} \cap \{p', q'\} = \varnothing$, and
\item $wt_{p'q'}t_{pq} \gcov wt_{p'q'} \gcov w$. 
\end{itemize}
Then $m_{pq}(wt_{p'q'})=m_{pq}(w)$, $m'_{qp}(wt_{p'q'})=m'_{qp}(w)$ and $wt_{pq} \gcov w$. 
\end{lem}


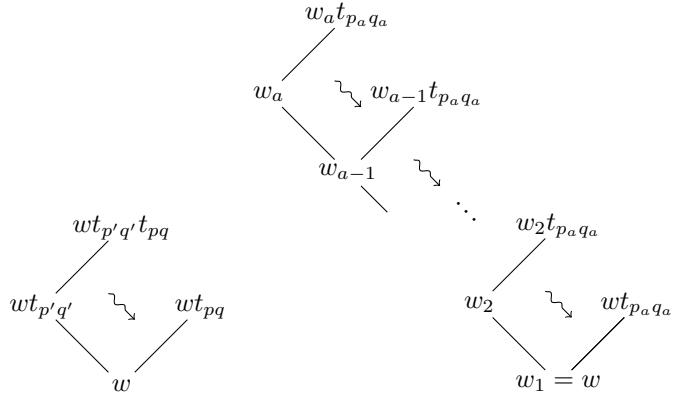
\begin{center}
\begin{tikzpicture}[scale=0.35]
\draw (-3,3)--(-1,5);
\draw (0,0)--(2,2);
\draw (-3,2)--(-1,0);
\draw (-0.5,-0.5) node{$w$};
\draw (-3.5,2.5) node{$wt_{p'q'}$};
\draw (-0.5,5.5) node{$wt_{p'q'}t_{pq}$};
\draw (2.5,2.5) node{$wt_{pq}$};

\draw [->,decorate,decoration={snake,amplitude=.4mm,segment length=2mm,post length=1mm}](-1,3)--(0,2);
\end{tikzpicture}
\begin{tikzpicture}[scale=0.35]
\newcount\hoge
\hoge=8
\draw (-3-\hoge,3+\hoge)--(-1-\hoge,5+\hoge);
\draw (-\hoge,\hoge)--(2-\hoge,2+\hoge);
\draw (-3,3)--(-1,5);
\draw (0,0)--(2,2);
\draw (-3,2)--(-1,0);
\draw (-3-\hoge,2+\hoge)--(-1-\hoge,0+\hoge);
\draw (-3-\hoge+3,2+\hoge-3)--(-1-\hoge+3-1,0+\hoge-3+1);

\draw (-0-\hoge/2,2+\hoge/2) node{$\cdot$};
\draw (-0-\hoge/2-0.3,2+\hoge/2+0.3) node{$\cdot$};
\draw (-0-\hoge/2+0.3,2+\hoge/2-0.3) node{$\cdot$};
\draw (0,0)--(2,2);
\draw (-0.5,-0.5) node{$w_1=w$};
\draw (2.5,2.5) node{$wt_{p_aq_a}$};
\draw (-3.5,2.5) node{$w_2$};
\draw (-0.5,5.5) node{$w_2t_{p_aq_a}$};
\draw (-0.5-\hoge,-0.5+\hoge) node{$w_{a-1}$};
\draw (2.5-\hoge,2.5+\hoge) node{$w_{a-1}t_{p_aq_a}$};
\draw (-3.5-\hoge,2.5+\hoge) node{$w_a$};
\draw (-0.5-\hoge,5.5+\hoge) node{$w_at_{p_aq_a}$};

\draw [->,decorate,decoration={snake,amplitude=.4mm,segment length=2mm,post length=1mm}](-1,3)--(0,2);
\draw [->,decorate,decoration={snake,amplitude=.4mm,segment length=2mm,post length=1mm}](-1-\hoge,3+\hoge)--(0-\hoge,2+\hoge);
\draw [->,decorate,decoration={snake,amplitude=.4mm,segment length=2mm,post length=1mm}](-1-\hoge+3,3+\hoge-3)--(0-\hoge+3,2+\hoge-3);

\end{tikzpicture}
\end{center}


\begin{proof}
Let us begin with a simple observation: suppose there exist two rectangles $R_1$ and $R_2$ with edges parallel to coordinate axes. Suppose that no two edges of these rectangles lie on the same line. Then, checking all the possibilities we see that
\begin{align*}
&\#(\text{NW and SE corners of $R_1$ lying inside $R_2$}) - \#(\text{NE and SW corners of $R_1$ lying inside $R_2$}) \\
&=\#(\text{NW and SE corners of $R_2$ lying inside $R_1$}) - \#(\text{NE and SW corners of $R_2$ lying inside $R_1$}). 
\end{align*}

First consider the case $R_1=[p, q] \times [w(p), w(q)]$ and $R_2=[p', q'] \times [w(p'), w(q')]$ in the observation above. $wt_{p'q'}t_{pq} \gcov wt_{p'q'} \gcov w$ implies that the first term in the left-hand side and the second term in the right-hand side vanish (here the coordinate system is taken so that points with smaller coordinates go NW). Thus all the terms must vanish. In particular the first term on the right-hand side vanishes and thus $wt_{pq} \gcov w$. 

We have shown that none of the points $(p, w(p)), (p, w(q)), (q, w(p))$ and $(q, w(q))$ lie in $[p', q'] \times [w(p'), w(q')]$. 
Since $m_{pq}(w)$ (resp.\ $m_{p'q'}(w)$) is the number of points of the graph of $w$ 
lying inside the rectangle $R'_1 = [q, M] \times [w(p),w(q)]$ (resp.\ $R'_1 = [-M, p] \times [w(p),w(q)]$) for $M \gg 0$ and the graphs of $w$ and $wt_{p'q'}$ differ only at the vertices of the rectangle $R_2 = [p',q'] \times [w(p'),w(q')]$, applying the observation to these rectangles shows the remaining claims.

\end{proof}

\begin{proof}[Proof of Proposition \ref{mainprop}]
\ 

\emph{Proof for $X_{i,d}(w)$}: 
We assume $(w\zeta)^{-1} \lex\leq (w\zeta')^{-1}$ and show that $\phi_{\zeta'}(v_{\zeta})=0$ unless $\zeta' = \zeta$ and $\phi_{\zeta}(v_{\zeta})$ is a nonzero multiple of $u_{w\zeta}$. 
Let $\zeta = t_{p_1q_1} \cdots t_{p_dq_d}$ and $\zeta'=t_{p'_1q'_1} \cdots t_{p'_dq'_d}$ as in Definition \ref{defn-pieri}. We write $w_a = w t_{p_1q_1} \cdots t_{p_{a-1}q_{a-1}}$ and $w'_a = w t_{p'_1 q'_1} \cdots t_{p'_{a-1}q'_{a-1}}$.  

For $\zeta = \prod_j t_{p_jq_j}$ and $\zeta' = \prod_j t_{p'_jq'_j}$ in $X_{i,d}(w)$ we have
\[
\phi_{\zeta'}(v_\zeta) = \sum_{\sigma \in S_d} \left( u_{p_{\sigma(d)}q'_{d}} \wedge \cdots \wedge u_{p_{\sigma(1)}q'_{1}} \wedge
( \prod_{j=1}^d E_j \prod_{j=1}^d E'_j \cdot u_w ) \right)
\quad \cdots (*)
\]
where $E_j = e_{p_jq_j}^{m_{p_jq_j}(w_j)}$ and $E'_j = (e'_{q'_jp'_j})^{m'_{q'_jp'_j}(w'_j)}$.

If $w(p_1) < w(p'_1)$, then $(w\zeta)^{-1}(w(p_1)) = q_1 > p_1 = (w\zeta')^{-1}(w(p_1))$ and $(w\zeta)^{-1}(j) = w^{-1}(j) = (w\zeta')^{-1}(j)$ for all $j < w(p_1)$, and this contradicts the assumption $(w\zeta)^{-1} \lex\leq (w\zeta')^{-1}$. Thus $w(p_1) \geq w(p'_1)$. 
Also, by a similar argument, if $p_1 = p'_1$ then $q_1 \leq q_1'$. 

Fix $\sigma \in S_d$. Let $1 \leq a \leq d$ be minimal such that $p_a = p_{\sigma(1)}$. Note that this in particular implies $w_a(p_a) = w(p_a)$. 
We have
\begin{align*}
& u_{p_{\sigma(d)}q'_{d}} \wedge \cdots \wedge u_{p_{\sigma(1)}q'_{1}} \wedge
(\prod_j E_j \prod_j E'_j \cdot u_w)\\
&= u_{p_{\sigma(d)}q'_{d}} \wedge \cdots \wedge u_{p_{\sigma(2)}q'_{2}} \wedge
\prod_{j \neq a} E_j \prod_{j \neq 1} E'_j \cdot (u_{p_{\sigma(1)} q'_1} \wedge E_a E'_1 u_w) \\
&= u_{p_{\sigma(d)}q'_{d}} \wedge \cdots \wedge u_{p_{\sigma(2)}q'_{2}} \wedge
\prod_{j \neq a} E_j \prod_{j \neq 1} E'_j \cdot (u_{p_a q'_1} \wedge E_a E'_1 u_w) \\
&= u_{p_{\sigma(d)}q'_{d}} \wedge \cdots \wedge u_{p_{\sigma(2)}q'_{2}} \wedge
\prod_{j \neq a} E_j \prod_{j \neq 1} E'_j \cdot (u_{p_a q'_1} \wedge e_{p_aq_a}^{m_{p_aq_a}(w)} (e'_{q'_1p'_1})^{m'_{q'_1p'_1}(w)} u_w)
\end{align*}
where the last equality is by Lemma \ref{m-invar} (note that $w'_1 = w$ by definition). 

First consider the case $w(p_1) > w(p'_1)$. We show that the summand in $(*)$ vanishes for all $\sigma$. 
It suffices to show $u_{p_a q'_1} \wedge e_{p_aq_a}^{m_{p_aq_a}(w)} (e'_{q'_1p'_1})^{m'_{q'_1p'_1}(w)} u_w = 0$. 
We have $w(p_a) = w_a(p_a) \geq w(p_1) > w(p'_1)$. 
Thus by Lemma \ref{pqpq} we see $u_{p_a q'_1} \wedge e_{p_aq_a}^{m_{p_aq_a}(w)} (e'_{q'_1p'_1})^{m'_{q'_1p'_1}(w)} u_w = 0$ (note that $wt_{p_aq_a} \gcov w$ by Lemma \ref{m-invar}). 

Next consider the case $w(p_1) = w(p'_1)$ and $a > 1$. 
In this case we see $u_{p_a q'_1} \wedge e_{p_aq_a}^{m_{p_aq_a}(w)} (e'_{q'_1p'_1})^{m'_{q'_1p'_1}(w)} u_w = 0$ since $w(p_a)=w_a(p_a)>w(p_1) = w(p'_1)$.

Next consider the case $w(p_1) = w(p'_1)$, $a=1$ and $q_1 < q_1'$. 
Then since $wt_{p_1q_1}, wt_{p'_1q'_1} \gcov w$ it follows that $w(q'_1) < w(q_1)$. So again by Lemma \ref{pqpq} we see $u_{p_a q'_1} \wedge e_{p_aq_a}^{m_{p_aq_a}(w)} (e'_{q'_1p'_1})^{m'_{q'_1p'_1}(w)} u_w = 0$. 

So the only remaining summands in $(*)$ are the ones with $(p_1, q_1) = (p'_1, q'_1)$ and $a=1$, i.e.\ $p_{\sigma(1)}=p_1$. It is easy to see that the sum of such summands is a nonzero multiple of the sum of terms with $\sigma(1)=1$. 
If $\sigma(1)=1$ we have, by the latter part of Lemma \ref{pqpq}, 
\begin{align*}
& u_{p_{\sigma(d)}q'_{d}} \wedge \cdots \wedge u_{p_{\sigma(1)}q'_{1}} \wedge
(\prod_{j=1}^d E_j \prod_{j=1}^d E'_j \cdot u_w) \\
&= u_{p_{\sigma(d)}q'_{d}} \wedge \cdots \wedge u_{p_{\sigma(2)}q'_{2}} \wedge
\prod_{j=2}^d E_j \prod_{j=2}^d E'_j \cdot (u_{p_1 q_1} \wedge e_{p_1q_1}^{m_{p_1q_1}(w)} (e'_{q_1p_1})^{m'_{q_1p_1}(w)} u_w) \\
&= (\text{nonzero const.}) \cdot u_{p_{\sigma(d)}q'_{d}} \wedge \cdots \wedge u_{p_{\sigma(2)}q'_{2}} \wedge
(\prod_{j=2}^d E_j \prod_{j=2}^d E'_j \cdot u_{wt_{p_1q_1}}). 
\end{align*}
So, working inductively on $d$ (using $wt_{p_1q_1}$, $t_{p_2q_2} \cdots t_{p_dq_d}$ and $t_{p'_2q'_2}\cdots t_{p'_dq'_d}$ in place of $w$, $\zeta$ and $\zeta'$ respectively, noting that if $(p_1, q_1) = (p'_1, q'_1)$ then $(w\zeta)^{-1} \lex\leq (w\zeta')^{-1}$ implies
$((wt_{p_1q_1}) \cdot t_{p_2q_2} \cdots t_{p_dq_d})^{-1} = (w\zeta)^{-1}
\lex\leq
(w\zeta')^{-1} = ((wt_{p_1q_1}) \cdot t_{p'_2q'_2} \cdots t_{p'_dq'_d})^{-1}$) we see that:
\begin{itemize}
\item $u_{p_{\sigma(d)}q'_{d}} \wedge \cdots \wedge u_{p_{\sigma(1)}q'_{1}} \wedge
(\prod_j E_j \prod_j E'_j \cdot u_w)$ vanishes if $(w\zeta)^{-1} \lex< (w\zeta')^{-1}$, or if $\zeta'=\zeta$ and $\sigma \neq \id$, 
and 
\item if $\zeta'=\zeta$ and $\sigma=\id$ then it is a nonzero multiple of $u_{w\zeta}$. 
\end{itemize}
This finishes the proof for $X_{i,d}(w)$. 

\emph{Proof for $Y_{i,d}(w)$}: 
This proceeds much similarly to the previous case. 
Here instead of $(*)$ we use
\begin{align*}
\phi_{\zeta'}(v_\zeta) &= \sum_{\sigma \in S_d} \left( \sgn(\sigma) \cdot u_{p_{\sigma(d)}q'_{d}} \wedge \cdots \wedge u_{p_{\sigma(1)}q'_{1}} \wedge
(\prod_{j=1}^d E_j \prod_{j=1}^d E'_j \cdot u_w) \right) \\
&= \sum_{\sigma \in S_d} \left( u_{p_{d}q'_{\sigma^{-1}(d)}} \wedge \cdots \wedge u_{p_{1}q'_{\sigma^{-1}(1)}} \wedge
(\prod_{j=1}^d E_j \prod_{j=1}^d E'_j \cdot u_w) \right)
\end{align*}
where $E_j = e_{p_jq_j}^{m_{p_jq_j}(w_j)}$ and $E'_j = (e'_{q'_jp'_j})^{m'_{q'_jp'_j}(w'_j)}$ as before. 

We assume $(w\zeta)^{-1} \rlex\leq (w\zeta')^{-1}$. 
Fix $\sigma$ and take $1 \leq a \leq d$ minimal with $q'_a = q'_{\sigma^{-1}(1)}$. 
By an argument similar to the above, it suffices to show that
$u_{p_1q'_a} \wedge e_{p_1q_1}^{m_{p_1q_1}(w)} (e'_{q'_ap'_a})^{m'_{q'_ap'_a}(w)} u_w$ is zero unless $a=1$ and $(p'_1, q'_1) = (p_1, q_1)$, and in a such case it is a nonzero multiple of $u_{wtp_1q_1}$. 

Since $(w\zeta)^{-1} \rlex\leq (w\zeta')^{-1}$ by the hypothesis, we see that $w(q_1) \geq w(q'_1)$, and that if $w(q_1)=w(q'_1)$ then $p_1 \leq p'_1$. 

If $w(q_1)>w(q'_1)$ then the claim follows from Lemma \ref{pqpq} since $w(q_1) > w(q'_1) \geq w'_a(q'_a)=w(q'_a)$. If $w(q_1)=w(q'_1)$ and $a > 1$ then the claim follows from Lemma \ref{pqpq} since in this case $w(q_1) = w(q'_1) > w(q'_a)$ by $wt_{p_1q_1}, wt_{p'_1q'_1} \gcov w$. 
If $q_1=q'_1$, $a = 1$ and $p_1 < p'_1$ the claim follows from Lemma \ref{pqpq} since $w(p_1) > w(p'_1)$. 
Finally if $(p_1, q_1) = (p'_1, q'_1)$ and $a=1$ then $u_{p_1q'_a} \wedge e_{p_1q_1}^{m_{p_1q_1}(w)} (e'_{q'_ap'_a})^{m'_{q'_ap'_a}(w)} u_w = u_{p_1q_1} \wedge e_{p_1q_1}^{m_{p_1q_1}(w)} (e'_{q_1p_1})^{m'_{q_1p_1}(w)} u_w$ is a constant multiple of $u_{wt_{p_1q_1}}$ by Lemma \ref{pqpq}.  
\end{proof}

\end{document}